\documentclass[12pt]{amsart}
\usepackage{fullpage}

\usepackage{amsfonts,amssymb,amscd,amsmath,enumerate,verbatim,color,xcolor}
\usepackage[latin1]{inputenc}
\usepackage{amscd}
\usepackage{latexsym}
\usepackage{tikz}
\usepackage{graphicx} %======> for arxiv!!
\usepackage[normalem]{ulem}
\usepackage{mathptmx}

\usepackage{hyperref}
\hypersetup{colorlinks,linkcolor=blue ,citecolor=blue, urlcolor=blue}
\def\NZQ{\mathbb}               % the font for N,Z,Q,R,C

\def\ZZ{{\NZQ Z}}
\def\RR{{\NZQ R}}

\def\frk{\mathfrak}               % font for "Fraktur"

\def\Phi{{\frk N}}
\def\ab{{\mathbf a}}

\def\eb{{\mathbf e}}
\def\tb{{\mathbf t}}

\def\xb{{\mathbf x}}

\def\opn#1#2{\def#1{\operatorname{#2}}} % to make operators
\opn\gr{gr}

\def\Mc{{\mathcal M}}

\def\Pc{{\mathcal P}}

\newtheorem{Theorem}{Theorem}[section]
\newtheorem{Lemma}[Theorem]{Lemma}

\newtheorem{Proposition}[Theorem]{Proposition}

\theoremstyle{definition}

\newtheorem{Example}[Theorem]{Example}

\newtheorem{Problem}[Theorem]{Problem}

\newtheorem{Question}[Theorem]{Question}

\let\epsilon\varepsilon
\let\phi=\varphi
\let\kappa=\varkappa
\opn\dis{dis}
\opn\height{height}
\opn\dist{dist}
\def\pnt{{\raise0.5mm\hbox{\large\bf.}}}

\opn\Lex{Lex}
\opn\conv{conv}
\def\codeg#1{{\rm codeg}(#1)}
\def\deg#1{{\rm deg}(#1)}
\def\reg#1{{\rm reg}(#1)}

\usepackage{graphicx} % Required for inserting images

\title{Codegree and regularity of stable set polytopes}
\author{Koji Matsushita}
\address{Koji Matsushita, Department of Pure and Applied Mathematics, Graduate School of Information Science and Technology, Osaka University, Suita, Osaka 565-0871, Japan}
\email{k-matsushita@ist.osaka-u.ac.jp}

\author{Akiyoshi Tsuchiya}
\address{Akiyoshi Tsuchiya,
Department of Information Science,
Faculty of Science,
Toho University,
2-2-1 Miyama, Funabashi, Chiba 274-8510, Japan} 
\email{akiyoshi@is.sci.toho-u.ac.jp}
\keywords{stable set polytope, codegree, clique number, chromatic number, matching polytope, perfect graph, $h$-perfect graph, toric ring, regularity}
\subjclass[2020]{52B20, 05C15, 13C10}

\date{}

\begin{document}

\maketitle
\begin{abstract}
The codegree $\codeg{\Pc}$ of a lattice polytope $\Pc$ is a fundamental invariant in discrete geometry. In the present paper, we investigate the codegree of the stable set polytope $\Pc_G$ associated with a simple graph $G$. Specifically, we establish the inequalities
\[
\omega(G) + 1 \leq \codeg{\Pc_G} \leq \chi(G) + 1,
\]
where $\omega(G)$ and $\chi(G)$ denote the clique number and the chromatic number of $G$, respectively. 
Furthermore, an explicit formula for $\codeg{\Pc_G}$ is given when $G$ is either a line graph or an $h$-perfect graph. Finally, as an application of these results, we provide upper and lower bounds on the regularity of the toric ring associated with $\Pc_G$.
\end{abstract}
\section{Introduction}
A \textit{lattice polytope} 
is a convex polytope all of whose vertices have integer coordinates. 
Let $\Pc \subset \RR^n$ be a full-dimensional lattice polytope. Then the \textit{codegree} of $\Pc$, denoted by $\codeg{\Pc}$, is the smallest positive integer $k$ such that the $k$th dilated polytope $k\Pc$ of $\Pc$ has an interior lattice point. 
On the other hand, the \textit{degree} of $\Pc$, denoted by $\deg{\Pc}$, is defined by $\deg{\Pc}=n+1-\codeg{\Pc}$.
Then $\deg{\Pc}$  coincides with the degree of the $h^*$-polynomial of $\Pc$, which encodes the number of lattice points in the dilatations of $\Pc$.
The codegree and the degree are fundamental invariants of lattice polytopes, playing significant roles in various fields, particularly in discrete geometry and commutative algebra.
See, e.g., \cite{BeckRobins2015book} for detailed information for the degree and codegree of lattice polytopes.

Throughout the present paper, we only treat simple graphs, which we simply call graphs instead of simple graphs.
The \textit{stable set polytope} $\Pc_G$ associated with a graph $G$ is the convex hull of the indicator vectors of all stable sets of $G$, where a set of vertices of $G$ is called \textit{stable} if no two of its elements are adjacent.
In the present paper, we investigate the codegree of $\Pc_G$.
The following is the first main result:
\begin{Theorem}\label{thm:bound_codeg}
    Let $G$ be a graph.
Then one has
\[
\omega(G)+1 \leq \codeg{\Pc_G} \leq \chi(G)+1,
\]
where $\omega(G)$ and $\chi(G)$ denote the clique number and the chromatic number of G, respectively.
In particular, if $G$ is perfect, then we obtain
\[
\omega(G)+1 = \codeg{\Pc_G} = \chi(G)+1.
\]
\end{Theorem}
It is well known that $\omega(G) \leq \chi(G)$ for any graph $G$. From the above theorem, we can regard $\codeg{\Pc_G} - 1$ as an invariant lying between these two important invariants in graph theory. Furthermore, it is an important fact that there exists a graph $G$ satisfying $\omega(G) < \codeg{\Pc_G} - 1 < \chi(G)$ (see Example \ref{exam:inequalities}).

Also, Theorem \ref{thm:bound_codeg} implies that $\codeg{\Pc_G}$ can be computed by using invariants of the underlying graph $G$ when $G$ is a perfect graph. 
Next, for several classes of graphs, we give an explicit formula of $\codeg{\Pc_G}$ in terms of $G$.
The \textit{matching polytope} $\Mc_G$ of $G$ is the convex hull of the indicator vectors of all matchings of $G$, where a set of edges of $G$ is called a \textit{matching} if no two of its elements share a common endpoint.
We can regard $\Mc_G$ as the stable set polytope of the line graph of $G$. Namely, $\Mc_G=\Pc_{L(G)}$, where $L(G)$ denotes the line graph of $G$. Then an explicit formula for $\codeg{\Mc_G}$ is given as follows:
\begin{Theorem}\label{thm:matching_polytope}
    Let $G$ be a graph with connected components $G_1,\ldots,G_r$.
    Then one has
\[
\codeg{\Mc_G}=\begin{cases}
    \Delta(G)+2 & \left(
   % \begin{array}{ll}
    \mbox{if $\Delta(G)$ is even and each $G_i$ with $\Delta(G_i)=\Delta(G)$ is a complete graph}
%        \mbox{all connected components of $G$ such that each contain a vertex} \\ \mbox{whose degree is equal to $\Delta(G)$ are $K_m$ for some odd integer $m$}
    %\end{array}
    \right),\\
    \Delta(G)+1 & (\mbox{otherwise}),
\end{cases}
\]
where 
 $\Delta(G)$ %and $K_m$ 
 denotes the maximal degree of $G$.% and the complete graph with $m$ vertices, respectively.
%and
%\[
%\deg{\Mc_G}=\begin{cases}
%    |E(G)|-\deg{G}-1 & (\mbox{$n$ is odd and $G$ is a complete graph}),\\
%    |E(G)|-\deg{G} & (\mbox{otherwise}).
%\end{cases}
%\]
\end{Theorem}
Note that $\omega(L(G))=\Delta(G)$ unless $G$ is a triangle, and if $G$ is a complete graph with an odd number of vertices $\geq 5$, then one has
\[
\Delta(G)+1=\omega(L(G))+1 < \codeg{\Pc_{L(G)}} = \chi(L(G))+1=\Delta(G)+2.
\]

A graph is said to be \textit{$h$-perfect} if $\Pc_G$ is defined by the constraints corresponding to cliques and odd holes, and the nonnegativity constraints. Note that every perfect graph is $h$-perfect.
Then for an $h$-perfect graph $G$, an explicit formula for $\codeg{\Pc_G}$ is given as follows:

\begin{Theorem}\label{thm:h-perfect}
    Let $G$ be an $h$-perfect graph.
Then one has
\begin{align*}
\codeg{\Pc_G}&=\omega(G)+1.
%\deg{\Pc_G}&=n-\omega(G).  
\end{align*}
\end{Theorem}
Note that if $G$ is an odd cycle of length $\geq 5$, then $G$ is $h$-perfect and  one has
\[
3=\omega(G)+1=\codeg{\Pc_G}< \chi(G)+1=4.
\]

Finally, we discuss the regularity of the toric ring of $\Pc_G$.
In combinatorial commutative algebra, finding methods to evaluate algebraic invariants using combinatorial invariants is an important problem. For example, in \cite{HaVanTuel} and \cite{Katzman}, the regularity of edge ideals is evaluated using invariants of the underlying graphs. Similarly, in this paper, we establish upper and lower bounds for the regularity of the toric rings of stable set polytopes in terms of invariants of the underlying graphs.
For a lattice polytope $\Pc$, let $K[\Pc]$ be the toric ring associated with $\Pc$ over a field $K$ and $\reg{K[\Pc]}$ the regularity of $K[\Pc]$. As an application of Theorem \ref{thm:bound_codeg}, we obtain the following.
\begin{Theorem}\label{thm:reg}
    Let $G$ be a graph with $n$ vertices.
    Then one has
    \[
n-\chi(G) \leq \reg{K[\Pc_G]}.
    \]
    Moreover, if $K[\Pc_G]$ is normal, then we obtain
    \[
\reg{K[\Pc_G]} \leq n-\omega(G). 
    \]
    In particular, if $G$ is perfect, then one has
        \[
n-\chi(G)=\reg{K[\Pc_G]} = n-\omega(G). 
    \]
\end{Theorem}

The present paper is organized as follows:
In Section \ref{sect:bound}, after recalling the definitions and properties of stable set polytopes and perfect graphs, we give a proof of Theorem \ref{thm:bound_codeg}. Section \ref{sect:formula} gives explicit formulas for $\codeg{\Pc_G}$ in the cases where  $G$  is a line graph or an $h$-perfect graph. In particular, Section \ref{subsect:matching} proves Theorem \ref{thm:matching_polytope}, and Section \ref{subsect:h-perfect} proves Theorem \ref{thm:h-perfect}. Section \ref{sect:reg} discusses the regularity of  $K[\Pc_G]$ , with a focus on proving Theorem \ref{thm:reg}. Finally, Section \ref{sect:exam} considers examples related to the inequalities in Theorem \ref{thm:bound_codeg}.

\section{bounds on the codegree of stable set polytopes}
\label{sect:bound}
In this section, a proof of Theorem \ref{thm:bound_codeg} is given.
First, we recall the definition of stable set polytopes.
Let $G$ be a graph on $[n]:=\{1,2,\ldots,n\}$ with edge set $E(G)$.
A subset $S \subset [n]$ is called a \textit{stable set} (or an \textit{independent set}) of $G$ if $\{i,j\} \notin E(G)$ for all $i,j \in S$ with $i \neq j$. In particular, the empty set $\emptyset$ and any singleton $\{i\}$ with $i \in [n]$ are stable sets of $G$.
Let $S(G)$ be the set of all stable sets of $G$. For a subset $A \subset [n]$, we denote $\eb_A \in \RR^n$ the indicator vector of $A$, namely, $\eb_A=\sum_{i \in A} \eb_i$, where $\eb_i$ is the $i$th unit coordinate vector in $\RR^n$.  Remark that $\eb_{\emptyset}=\mathbf{0}:=(0,\ldots,0)\in \RR^n$. Then the \textit{stable set polytope} of $G$ is defined as
\[
\Pc_G=\conv\{\eb_S: S \in S(G)\} \subset \RR^n.
\]
Note that $\dim \Pc_G$ is equal to the number of vertices of $G$.

Next, we define a perfect graph. A subset $C \subset [n]$ is said to be \textit{clique} of $G$ if for any $i, j \in C$ with $i \neq j$, $\{i,j\} \in E(G)$. We denote $\omega(G)$ the maximal cardinality of cliques of $G$, which is called the \textit{clique number} of $G$.
A map $f : [n] \to [k]$ with a positive integer $k$ is said to be a \textit{coloring} of $G$, if for any $\{i,j\} \in E(G)$, $f(i) \neq f(j)$.
The minimal integer $k$ such that there exists a coloring $f : [n] \to [k]$ of $G$ is called a \textit{chromatic number} of $G$.
In general, one has
\[
\omega(G) \leq \chi(G).
\]
However, the equality does not always hold (e.g., an odd cycle of length $\geq 5$).
A graph $G$ is called \textit{perfect} if for any induced subgraph $H$ of $G$, $\omega(H)=\chi(H)$.
Perfect graphs were introduced by Berge in \cite{Berge1960}. A \textit{hole} is an induced cycle of length $\geq 5$ and an \textit{antihole} is the complement graph of a hole. In 2006, Chudnovsky, Robertson, Seymour, and Thomas solved a famous conjecture in graph theory, which was conjectured by Berge and is now known as the strong perfect graph theorem.
\begin{Proposition}[\cite{strongperfect}]
\label{prop:strong_perfect}
    A graph is perfect if and only if it has no odd holes and no odd antiholes.
\end{Proposition}

Also, there is a polyhedral characterization of perfect graphs.

\begin{Proposition}[{\cite[Theorem~3.1]{Chvatal1975}}]
\label{prop:poly_perfect}
Let $G$ be a graph on $[n]$. Then $G$ is perfect if and only if
\begin{align}\label{ineq:perfect}
    \Pc_G 
    =
    \left\{ 
    \xb \in \RR^n \; : \; 
    \begin{array}{rlc}
        x_i &\geq 0, & \forall i \in [n] \vspace{0.2cm}\\
        \displaystyle \sum_{i \in Q} x_i &\leq 1, &  \text{ for any maximal clique $Q$} 
    \end{array}
    \right\}.
\end{align}
\end{Proposition}
Note that for a (not necessarily perfect) graph $G$, each inequality appearing in (\ref{ineq:perfect}) defines a facet of $\Pc_G$ (\cite{PadbergPackingpolyhedra}).

Before proving Theorem \ref{thm:bound_codeg}, we show the following lemma.

\begin{Lemma}\label{lem:int_point}
    Let $G$ be a graph on $[n]$.
For any $k\in \ZZ_{>0}$, one has ${\rm int}(k\Pc_G) \cap \ZZ^n \neq \emptyset$ if and only if $\eb_{[n]} \in {\rm int}(k\Pc_G)$. In particular, one has 
\[
\codeg{\Pc_G}=\min( k \in \ZZ_{> 0} : \eb_{[n]} \in {\rm int}(k\Pc_G)).
\]
\end{Lemma}

\begin{proof}
   Since subsets of stable sets are also stable, for any inequality $\sum_{i=1}^na_ix_i \leq b$ defining a facet of $\Pc_G$, one has 
   \begin{enumerate}
        \item[(i)] $b=0$ and $(a_1,\ldots,a_n)=-\eb_i$ for some $i\in [n]$ (called a \textit{trivial facet}), or 
        \item[(ii)] $b>0$ and $(a_1,\dots,a_n)\in \ZZ_{\geq 0}^n$.
    \end{enumerate}
    
    Suppose that there is a lattice point $(c_1,\ldots,c_n)$ in ${\rm int}(k\Pc_G)$. 
    To see that $\eb_{[n]} \in {\rm int}(k\Pc_G)$, it is enough to show that $\eb_{[n]}$ satisfies any inequality of the form (ii) strictly, that is, $\sum_{i=1}^na_i<kb$.

    Since $(c_1,\ldots,c_n)\in {\rm int}(k\Pc_G)$, one has $(c_1,\ldots,c_n)\in \ZZ_{>0}^n$ and $\sum_{i=1}^na_ic_i<kb$.
    Therefore, we get $\sum_{i=1}^na_i\leq \sum_{i=1}^na_ic_i <kb$, as desired.

    The converse is trivial.
\end{proof}

Next, we prove the lower bound of Theorem~\ref{thm:bound_codeg}.

\begin{Proposition}\label{prop:lower_bound}
    Let $G$ be a graph on $[n]$.
    Then one has $\codeg{\Pc_G}\geq \omega(G)+1$. Moreover, one has $\codeg{\Pc_G}=\omega(G)+1$ if $G$ is perfect.
\end{Proposition}

\begin{proof}
    For a  clique $Q$ of $G$ with $|Q|=\omega(G)$, the inequality $\sum_{i\in Q}x_i \leq 1$ defines a facet of $\Pc_G$ (\cite[Theorem~2.4]{PadbergPackingpolyhedra}).
    Thus, if $\eb_{[n]}\in {\rm int}(k\Pc_G)$ for some $k\in \ZZ_{>0}$, then $k$ must be greater than $\omega(G)$, which implies that $\omega(G)+1\leq \codeg{\Pc_G}$ from Lemma~\ref{lem:int_point}.

    Suppose that $G$ is perfect. Then any non-trivial facet of $\Pc_G$ is defined by the inequality $\sum_{i\in Q}x_i \leq 1$ for some maximal clique of $G$ (Proposition \ref{prop:poly_perfect}).
    Thus, we have $\codeg{\Pc_G}=\omega(G)+1$.
\end{proof}
Note that, for a perfect graph $G$, the equality $\codeg{\Pc_G} = \omega(G)+1$ is also shown in the proof of \cite[Theorem 2.1 (b)]{OhsugiHibiSpecialsimplicies}.

Before proceeding to the proof of Theorem \ref{thm:bound_codeg}, we recall the notion of the join of graphs. 
Let $G_1$ and $G_2$ be graphs on disjoint vertex sets. 
The \emph{join} of $G_1$ and $G_2$ is the graph obtained from the disjoint union of $G_1$ and $G_2$ 
by adding an edge between every vertex of $G_1$ and every vertex of $G_2$. 
The join of more than two graphs is defined analogously by repeatedly applying this construction.

Now, we turn to prove Theorem \ref{thm:bound_codeg}.

\begin{proof}[Proof of Theorem~\ref{thm:bound_codeg}]
    We have already shown that $\codeg{\Pc_G}\geq \omega(G)+1$ in Proposition~\ref{prop:lower_bound}, so we prove that $\codeg{\Pc_G}\leq \chi(G)+1$.
    By Lemma~\ref{lem:int_point}, it is enough to show that $\eb_{[n]}\in {\rm int}((\chi(G)+1)\Pc_G)$.
    % We take a coloring $f : [n] \to [\chi(G)]$ of $G$.
    % For each $j\in [\chi(G)]$, let $s_j:=\sum_{i\in f^{-1}(j)}\eb_i$ and let $\Pc:=\conv(\{{\bf 0},s_1,\ldots,s_{\chi(G)}\})$.
    % Since $f^{-1}(j)$ is a stable set of $G$, we have $\Pc\subset \Pc_G$.
    % Moreover, we can see that 
    % \[
    % \frac{1}{\chi(G)+1}({\bf 0}+s_1+\cdots+s_{\chi(G)})=\frac{1}{\chi(G)+1}\eb_{[n]}\in {\rm int}(\Pc).
    % \]
    % Therefore, $\eb_{[n]}\in {\rm int}((\chi(G)+1)\Pc)\subset {\rm int}((\chi(G)+1)\Pc_G)$.

    %\textcolor{blue}{
    We take a proper coloring $f : [n] \to [\chi(G)]$ of $G$.
    By the minimality of $\chi(G)$, we have $f^{-1}(i) \neq \emptyset$ for any $1 \leq i \leq \chi(G)$.
    For any $1 \leq i \leq \chi(G)$, let $G_i$ be the graph on $f^{-1}(i)$ with no edges, and let $G'$ be the join of $G_1,\ldots,G_{\chi(G)}$.
    Then since $G'$ is a complete multipartite graph, $G'$ is a perfect graph on $[n]$ with $\chi(G')=\chi(G)$. 
    Hence $\codeg{\Pc_{G'}}=\omega(G')+1=\chi(G')+1=\chi(G)+1$ from Proposition~\ref{prop:lower_bound}.
    By Lemma~\ref{lem:int_point}, this implies that $\eb_{[n]}\in {\rm int}((\chi(G)+1)\Pc_{G'}$.
    On the other hand, for each $i\in [\chi(G)]$, $f^{-1}(i)$ is a stable set of $G$. Moreover, the set $\{f^{-1}(1),\ldots,f^{-1}(\chi(G))\}$ coincides with the set of all maximal stable sets of $G'$.
    Hence $\Pc_{G'} \subset \Pc_{G}$. It then follows from $\dim \Pc_G=\dim \Pc_{G'}$ that ${\rm int}((\chi(G)+1)\Pc_{G'} \subset {\rm int}((\chi(G)+1)\Pc_G$. In particular, one has $\eb_{[n]} \in {\rm int}((\chi(G)+1)\Pc_G$.
    %}
\end{proof}

\section{Explicit formula for the codegree  of a stable set polytope}
\label{sect:formula}
In this section, for several classes of graphs, we give an explicit formula of $\codeg{\Pc_G}$ in terms of $G$.
In particular, proofs of Theorems \ref{thm:matching_polytope} and \ref{thm:h-perfect} are given.
\subsection{Matching polytopes}
\label{subsect:matching}
Let $G$ be a graph on $[n]$ with edge set $E(G)$.
A \textit{matching} $M$ of a graph $G$ is a subset of $E(G)$ such that every vertex of $G$ is incident to at most one edge in $M$.
The \textit{matching polytope} of $G$ is the convex hull of the indicator vectors of the matchings on $G$, that is,
    \[
    \Mc_G = \conv\left\{ \eb_M \in \RR^{E(G)} : M\subset E(G) \text{ is a matching of } G \right\}.
    \]
The matching polytope $\Mc_G$ is a stable set polytope of some graph.
Indeed, the \textit{line graph} $L(G)$ of $G$ is a graph whose vertex set is $E(G)$ and whose edge set is 
\[
\{ \{e, e'\} \subset \ E(G) : e \neq e' \mbox{ and } e \cap e' \neq \emptyset\}.
\]
Then one has $\Mc_G=\Pc_{L(G)}$ by changing coordinates.
    
The matching polytope has the following description:
\begin{Proposition}[{\cite{Edmonds1965}}]\label{prop:ineq_match}
Let $G$ be a graph on $[n]$ with edge set $E(G)$. Then one has 
\begin{align*}
    \Mc_G 
    =
    \left\{ 
    x \in \RR^{E(G)} \; : \; 
    \begin{array}{rlc}
        x(e) &\geq 0, & \forall e \in E(G) \vspace{0.2cm}\\
        \displaystyle \sum_{e \in \iota_G(v)} x(e) &\leq 1, &  \forall v \in [n] \\
        \displaystyle \sum_{e \in E(G_U)} x(e) &\leq \frac{|U| - 1}{2}, &  \forall U\subset [n] : \text{$|U|$ is odd}
    \end{array}
    \right\},
\end{align*}
where $\iota_G(v)$ and $G_U$ denote the set of all edges incident to a vertex $v$ in $G$ and the induced subgraph of $G$ with vertex set $U$, respectively.
\end{Proposition}

Now, we prove Theorem \ref{thm:matching_polytope}.

\begin{proof}[Proof of Theorem~\ref{thm:matching_polytope}]
   % Let $G_1,\ldots,G_r$ be connected components of $G$. 
    Since $\Mc_G=\Mc_{G_1}\times \cdots \times \Mc_{G_r}$, one has \[\codeg{\Mc_G}=\max\{\codeg{\Mc_{G_1}},\ldots,\codeg{\Mc_{G_r}}\}\]
by Lemma~\ref{lem:int_point}, where $G_1,\ldots,G_r$ are the connected components of $G$.
    Thus, we may assume that $G$ is connected and it suffices to prove that $\codeg{\Mc_G}=\Delta(G)+2$ if $G$ is a complete graph $K_m$ with $m$ vertices, where $m$ is odd, otherwise $\codeg{\Mc_G}=\Delta(G)+1$.
    
    By Proposition~\ref{prop:ineq_match}, if $\eb_{E(G)}\in {\rm int}(k\Mc_G)$ for some $k\in \ZZ_{>0}$, then the inequalities $|\iota_G(v)|<k$ and $|E(G_U)|<k\frac{|U|-1}{2}$ hold for any $v\in G$ and $U\subset [n]$ such that $|U|$ is odd.
    In particular, $k$ must be greater than $\Delta(G)$ for the first inequality to hold.

    On the other hand, it follows from the handshaking lemma and $|U|-\Delta(G_U)-1\geq0$ that
    \[
        |E(G_U)|=\frac{1}{2}\sum_{v\in U}|\iota_{G_U}(v)| \underset{(a)}{\leq} \frac{\Delta(G_U)\cdot |U|}{2} \underset{(b)}{\leq} \frac{(\Delta(G_U)+1)\cdot (|U|-1)}{2} \underset{(c)}{\leq} (\Delta(G)+1)\frac{|U|-1}{2}.
    \]
    The equalities of (a) and (b) hold if and only if $G_U=K_{|U|}$.
    In addition, the equality of (c) holds as well, that is,  we have $|E(G_U)|=(\Delta(G)+1)\frac{|U|-1}{2}$ if and only if $G_U=K_{|U|}=G$ since $G$ is connected.
    Therefore, if $G=K_m$ for some odd integer $m$, then $\codeg{\Mc_G}=\Delta(G)+2$, otherwise $\codeg{\Mc_G}=\Delta(G)+1$.
\end{proof}

\subsection{The stable set polytopes of  $h$-perfect graphs}
\label{subsect:h-perfect}
Let $G$ be a graph on $[n]$. We say that $G$ is \textit{$h$-perfect} \cite{h-perfect} if the stable set polytope has the following description:
\begin{align*}
    \Pc_G 
    =
    \left\{ 
    x \in \RR^n \; : \; 
    \begin{array}{rlc}
        x_i &\geq 0, & \forall i \in [n] \vspace{0.2cm}\\
        \displaystyle \sum_{i \in Q} x_i &\leq 1, &  \text{ for any clique $Q$} \\
        \displaystyle \sum_{i \in C} x_i &\leq \frac{|C| - 1}{2}, &  \text{ for any odd cycle $C$}
    \end{array}
    \right\}.
\end{align*}
From Proposition \ref{prop:poly_perfect} every perfect graph is $h$-perfect. 
However, an $h$-perfect graph is not necessarily perfect. For example, it is known that every odd cycle of length $\geq 5$ is $h$-perfect, but not perfect.

Now, we give a proof of Theorem \ref{thm:h-perfect}.
\begin{proof}[Proof of Theorem~\ref{thm:h-perfect}]
    If $\omega(G)=1$, that is, $G$ has no edges, then our assertion clearly holds, thus we may assume that $\omega(G)\geq 2$.

    By Proposition~\ref{prop:lower_bound}, we have $\codeg{\Pc_G}\geq \omega(G)+1$, so it suffices to show that $\eb_{[n]}\in {\rm int}((\omega(G)+1)\Pc_G)$.
    To this end, we show that $|C|<(\omega(G)+1)\frac{|C|-1}{2}$, equivalently $2<(\omega(G)-1)(|C|-1)$ for any odd cycle $C$ of $G$.
    This is true if $|C|\geq5$, and even if $|C|=3$, $\omega(G)\geq 3$ in this case, so it is true either way.
\end{proof}
\section{bounds on the regularities of the toric ring associated with $\Pc_G$}
\label{sect:reg}
Let $S = K[x_{1}, \ldots, x_{n}]$ denote the polynomial ring in $n$ variables over a field $K$ with each $\deg{x_i} = 1$.  Let $0 \neq I \subset S$ be a homogeneous ideal of $S$ and 
\[
0 \to \bigoplus_{j \geq 1} S(-j)^{\beta_{h, j}} \to \cdots \to \bigoplus_{j \geq 1} S(-j)^{\beta_{1, j}} \to S \to S/I \to 0
\]
a (unique) graded minimal free $S$-resolution of $S/I$.  The ({\em Castelnuovo-Mumford}) \textit{regularity} of $S/I$ is defined by
\[
\reg{S/I} = \max\{ j - i : \beta_{i, j} \neq 0 \}. 
\] 
See, e.g., \cite{BrunsHerzog} for detailed information about regularity.

In this section, we discuss the regularity of the toric rings associated with stable set polytopes.
In particular, a proof of Theorem \ref{thm:reg} is given.
Let $\Pc \subset \RR^n_{\geq 0}$ be a full-dimensional lattice polytope with $\Pc \cap \ZZ^n=\{\ab_1,\ldots,\ab_d\}$, and
let $K[\tb,s]:=K[t_1,\ldots,t_n,s]$ be the polynomial ring in $n+1$ variables  over a field $K$.
Given a nonnegative integer vector 
$\ab=(a_1,\ldots,a_n) \in  \ZZ_{\geq 0}^d$, we write 
$\tb^{\ab}:=t_1^{a_1} t_2^{a_2}\cdots t_n^{a_n} \in K[\tb,s]$.
The \textit{toric ring} of $\Pc$ is 
\[
K[\Pc]:=K[\tb^{\ab_1} s,\ldots,\tb^{\ab_d}s] \subset K[\tb,s]. 
\]
We regard $K[\Pc]$ as a homogeneous algebra by setting each $\deg{\tb^{\ab_i} s}=1$.
Let $R[\Pc]=K[x_{1},\ldots,x_{d}]$ denote the polynomial ring in $d$ variables over  $K$ with each $\deg{x_i}=1$.
The \textit{toric ideal} of $\Pc$ is the kernel of the surjective homomorphism 
$\pi:R[\Pc] \to K[\Pc]$ defined by $\pi(x_{i})=\tb^{\ab_i}s$ for $1 \leq i \leq d$.
Note that $I_{\Pc}$ is a prime ideal generated by homogeneous binomials.
See, e.g., \cite{HHObook} for details on toric rings and toric ideals.

We recall some properties for $K[\Pc]$.
A full-dimensional lattice polytope $\Pc \subset \RR^n$ is called \textit{spanning} if every lattice point in $\ZZ^n$ is affine integer combination of the lattice points in $\Pc$.
Moreover, we say that $\Pc \subset \RR^n$ has the \textit{integer decomposition property} if for every integer $k \geq 1$, every lattice point in $k\Pc$ is a sum of $k$ lattice points from $\Pc$.
A lattice polytope which has the integer decomposition property is called \textit{IDP}.
It is easy to see that an IDP polytope is spanning.
Furthermore, a spanning polytope $\Pc$ is IDP if and only if the toric ring $K[\Pc]$ is normal. In this case, $K[\Pc]$ is Cohen--Macaulay.

The regularity $\reg{K[\Pc]}:=\deg{R[\Pc]/I_{\Pc}}$ of $K[\Pc]$ does not have a direct combinatorial interpretation. However, if $\Pc$ is spanning, then $\deg{\Pc}$ gives a lower bound on $\reg{K[\Pc]}$.
\begin{Proposition}[{\cite[p. 5952]{HKN}}]
	\label{prop:span}
	Let $\Pc \subset \RR^{n}$ be a spanning lattice polytope. 
	Then one has $\reg{K[\Pc]} \geq \deg{\Pc}$. 
\end{Proposition}
On the other hand, it is known that if $K[\Pc]$ is Cohen--Macaulay, then $\reg{K[\Pc]}$ is equal to the degree of the $h$-polynomial of $K[\Pc]$ (\cite[Corollary 2.18]{HHObook}).
Moreover, if $\Pc$ is IDP, then $\deg{\Pc}$ coincides with the degree of the $h$-polynomial of $K[\Pc]$ (\cite[Theorem 4.5 and Lemma 4.22 (b)]{HHObook}). Hence we know the following.
\begin{Proposition}\label{prop:IDP}
	Let $\Pc \subset \RR^{n}$ be an IDP lattice polytope. 
	Then one has $\reg{K[\Pc]} = \deg{\Pc}$. 
\end{Proposition}

Now, we prove Theorem \ref{thm:reg}.
\begin{proof}[Proof of Theorem \ref{thm:reg}]
The origin ${\bf 0}$ and each unit coordinate vector $\eb_i$ are vertices of the stable set polytope $\Pc_G$. This implies that $\Pc$ is spanning.
Since $\deg{\Pc_G}=n+1-\codeg{\Pc_G}$, it then follows from Theorem \ref{thm:bound_codeg} and Proposition \ref{prop:span}
that 
\[
n+1-\chi(G)-1 = n-\chi(G) \leq \deg{\Pc_G} \leq \reg{K[\Pc_G]}.\]
On the other hand, if $\Pc$ is IDP, then from Theorem \ref{thm:bound_codeg}
and Proposition \ref{prop:IDP} one has
\[
\reg{K[\Pc_G]}=\deg{\Pc_G}=n+1-\codeg{\Pc_G} \leq n+1 -\omega(G)-1=n-\omega(G).
\]
If $G$ is perfect, then $\Pc_G$ is compressed, in particular, it is IDP (\cite{OhsugiHibiSquarefree}).
Hence we obtain
\[
 n-\chi(G)=\reg{K[\Pc_G]}=n-\omega(G),
\]
as desired.
\end{proof}
%It is an important problem to determine which stable set polytope (or matching polytope) has IDP.

%\begin{Proposition}
 %   Let $G$ be a graph containing neither $K_5 \setminus e$ nor $K_{3,3}$ as a minor.
%Then $\Mc_G$ has IDP.
%\end{Proposition}

%\begin{Corollary}
 %Then one has
%\[
%\reg{K[\Mc_G]}=n-\Delta(G).
%\]
%\end{Corollary}
In general, stable set polytope are not always IDP. For example, let $G$ be the graph whose complement graph is the disjoint union of two $5$-cycles. Then $\Pc_G$ is not IDP \cite[Proposition 8]{MOS}, in particular, $K[\Pc_G]$ is not normal. Hence we cannot apply Theorem \ref{thm:reg} to obtain an upper bound on $K[\Pc_G]$. However, in this case, by using Macaulay2 \cite{M2}, one has $\reg{K[\Pc_G]}=6$. Hence we obtain $\reg{K[\Pc_G]}=10-\omega(G)$.
Therefore, the following question naturally occurs.
\begin{Question}
    Let $G$ be a graph on $[n]$. Then does the following inequality hold?
    \[
\reg{K[\Pc_G]} \leq n-\omega(G).
    \]
\end{Question}

\section{Examples related with Theorem \ref{thm:bound_codeg}}
\label{sect:exam}
We return the inequalities in Theorem \ref{thm:bound_codeg}. Then there are the following 4 cases:
\begin{enumerate}[(i)]
    \item $\omega(G)+1=\codeg{\Pc_G} = \chi(G)+1$;
    \item $\omega(G)+1 < \codeg{\Pc_G} = \chi(G)+1$;
    \item $\omega(G)+1 = \codeg{\Pc_G} < \chi(G)+1$;
    \item $\omega(G)+1 < \codeg{\Pc_G} < \chi(G)+1$.
\end{enumerate}
For each case, is there a graph $G$ satisfying the condition?
We can give an example for each case.
\begin{Example}
\label{exam:inequalities}
{\rm 
(1) Every perfect graph satisfies the condition (i).

(2) Let $n \geq 5$ be an odd integer and set $G=L(K_n)$.%, where $K_n$ is a complete graph with $n$ vertices.
Then one has $\omega(G)=n-1, \chi(G)=n$ and $\Delta(G)=n-1$. Hence it follows from Theorem \ref{thm:matching_polytope} that $G$ satisfies the condition (ii).

(3) Let $n \geq 5$ be an odd integer and let $G=C_n$ be an odd cycle of length $n$.
Then one has $\omega(G)=2$ and $\chi(G)=3$. Since it is known that $C_n$ is $h$-perfect, from Theorem \ref{thm:h-perfect} $G$ satisfies the condition (iii).

(4) Assume that $C_5$ and $L(K_5)$ have no common vertices and let $G$ be the join of $C_5$ and $L(K_5)$.
Then one has $\omega(G)=\omega(C_5)+\omega(L(K_5))= 2+4=6$ and $\chi(G)=\chi(C_5)+\chi(L(K_5))=3+5=8$. Moreover, by using Normaliz \cite{normaliz}, we obtain $\codeg{\Pc_G}=8$. In particular, $(1,\ldots,1) \in \RR^{15}$} is an interior lattice point in $8\Pc_G$ from Lemma \ref{lem:int_point}.
Hence $G$ satisfies the condition (iv).
\end{Example}
As a preliminary question, one may ask whether the differences
\(\codeg{\Pc_G} - (\omega(G)+1)\) and \((\chi(G)+1) - \codeg{\Pc_G}\)
can be made arbitrarily large, either simultaneously or individually while keeping the other bounded or zero.
More general, we can consider the following problem.
\begin{Problem}
    Let $a,b,c$ be integers satisfying 
    \[
    0 < a \leq b \leq c.
    \]
    Determine when there exists a graph $G$ satisfying
    \begin{align*}
        \omega(G)+1=a,  \ \ \ \ \ \codeg{\Pc_G}=b, \ \ \ \ \  \chi(G)+1= c.
    \end{align*}
    In other words, 
determine all the possible sequences
\[
(\omega(G)+1, \codeg{\Pc_G}, \chi(G)+1).
\]
\end{Problem}
Note that the case $a=b=c$ in this problem can be solved by considering perfect graphs.
In investigating this problem, it seems useful to consider the join of graphs.
Indeed, for graphs $G_1$ and $G_2$ and their join $G$, one has
\[
\omega(G) = \omega(G_1) + \omega(G_2),  \chi(G) = \chi(G_1) + \chi(G_2).\]
Moreover, in Example \ref{exam:inequalities} (4), we have
\begin{align*}
\deg{\Pc_G}& = 8 = 3+5= \deg{\Pc_{C_5}} + \deg{\Pc_{L(K_5)}},\\
\codeg{\Pc_G}&= 3+6-1 =  \codeg{\Pc_{C_5}} + \codeg{\Pc_{L(K_5)}} - 1.
\end{align*}
However, in general these equalities for the degree and the codegree do not hold.
We present an example where these equalities fail.
\begin{Example}
    {\rm 
    Let $G$ be the join of two $5$-cycles.
    Then one has $\deg{\Pc_G}=5$ and $\codeg{\Pc_G}=6$.
    In particular, this example also satisfies the condition (iv) since $\omega(G)=4$ and $\chi(G)=6$.
    On the other hand, $\deg{\Pc_{C_5}}=3$ and $\codeg{\Pc_{C_5}}=3$. Hence we know that
\begin{align*}
\deg{\Pc_{C_5}}+\deg{\Pc_{C_5}}&=6 \neq 5 =\deg{\Pc_{G}},\\
\codeg{\Pc_{C_5}}+\codeg{\Pc_{C_5}}-1&=5 \neq 6 =\codeg{\Pc_{G}}.
\end{align*}}
\end{Example}
Thus, for graphs $G_1$ and $G_2$ and their join $G$, it is natural to ask
when the equalities
\begin{align*}
\deg{\Pc_G} &= \deg{\Pc_{G_1}} + \deg{\Pc_{G_2}},\\
\codeg{\Pc_G} &= \codeg{\Pc_{G_1}} + \codeg{\Pc_{G_2}} - 1
\end{align*}
hold.
A complete characterization of such pairs $(G_1,G_2)$ seems to be an interesting problem.

\subsection*{Acknowledgment}
The authors are grateful to the referees for a careful reading of the manuscript and for many helpful comments and suggestions.
The first author was partially supported by 
Grant-in-Aid for JSPS Fellows 22J20033. 
The second author was partially supported by 
JSPS KAKENHI 
%The third author was partially supported by JSPS KAKENHI 
22K13890.

\bibliographystyle{plain}
\bibliography{bibliography}
\end{document}